\documentclass[a4paper,10pt]{article}
\usepackage{amsmath}
\usepackage{ascmac}
\usepackage{amssymb}
\usepackage{multicol}
\usepackage{graphicx}
\usepackage{color}
\usepackage{amscd}
\usepackage{url}
\usepackage{comment}
\usepackage{bm}
\usepackage{theorem}
\usepackage{latexsym}
\usepackage{mathrsfs}
\usepackage{framed,color}
\usepackage{fancybox}
\usepackage{ulem}
\usepackage{hyperref}
\usepackage[all]{xy}
\theorembodyfont{\normalfont}
\newtheorem{theorem}{Theorem}[section]

\newtheorem{remark}[theorem]{Remark}
\newtheorem{example}[theorem]{Example}

\definecolor{shadecolor}{gray}{0.85}
\makeatletter
\newenvironment{proof}[1][\proofname]{\par
  \normalfont
  \topsep6\p@\@plus6\p@ \trivlist
  \item[\hskip\labelsep{\bfseries #1}\@addpunct{\bfseries}]\ignorespaces
}{
  \endtrivlist
}
\newcommand{\proofname}{\textit{Proof.}}

\newcommand{\ctext}[1]{\raise0.2ex\hbox{\textcircled{\scriptsize{#1}}}}

	\@addtoreset{equation}{section}

	\@addtoreset{figure}{section}

	\@addtoreset{table}{section}
	
\makeatother
\def\qed{\hfill$\Box$}
\title{On the boundedness of the geodesic curvature measure of a regular curve passing through a cross cap singularity}
\author{Kyoya Hashibori}
\date{\today}
\begin{document}

\maketitle

\begin{abstract}
In 2015, Hasegawa, Honda, Naokawa, Saji, Umehara, and Yamada defined intrinsic cross cap singularities, which are generalizations of cross cap singularities, and proved the Gauss-Bonnet type formula for surfaces without boundary that admit these singularities. In this paper, we prove the boundedness of geodesic curvature measures of a regular curve passing through an intrinsic cross cap singularity and generalize the Gauss-Bonnet type formula by Hasegawa et al. to the case of surfaces with boundary. Also, we give conditions for the boundedness of the geodesic curvature at an intrinsic cross cap singularity.
\footnote[1]{\textit{Methematics Subject Classification (2020)}. Primary 57R45; Secondary 53A05}
\end{abstract}

\section{Introduction}

A smooth map germ $f:\left(\mathbb{R}^2,\mathbf{0}\right)\to\left(\mathbb{R}^3,\mathbf{0}\right)$ is a \textit{cross cap} if $f$ is $\mathcal{A}$-equivalent to the map germ
\begin{equation}
f_0\left(u,v\right):=\left(u,uv,v^2\right)^T\label{1.1}
\end{equation}
at the origin $\mathbf{0}$, where $A^T$ is the transpose of $A$. Here, the two smooth map germs $f_i:\left(\mathbb{R}^2,\mathbf{0}\right)\to\left(\mathbb{R}^3,\mathbf{0}\right)\ \left(i=1,2\right)$ are \textit{$\mathcal{A}$-equivalent} at the origin $\mathbf{0}$ if there exist diffeomorphism germs $\varphi:\left(\mathbb{R}^2,\mathbf{0}\right)\to\left(\mathbb{R}^2,\mathbf{0}\right)$ and $\Phi:\left(\mathbb{R}^3,\mathbf{0}\right)\to\left(\mathbb{R}^3,\mathbf{0}\right)$ such that $f_2=\Phi\circ f_1\circ\varphi^{-1}$.

In \cite{1} West gave the following normal forms of cross caps in the Euclidean $3$-space $\mathbb{R}^3$: for a smooth map germ $f:\left(\mathbb{R}^2,\mathbf{0}\right)\to\left(\mathbb{R}^3,\mathbf{0}\right)$ with cross cap singularity at the origin $\mathbf{0}$, there exist an oriented local coordinate system $\left(u,v\right)$ centered at $\mathbf{0}$ and an orientation preserving isometry $T$ in $\mathbb{R}^3$ such that
\begin{eqnarray}
\widetilde{f}\left(u,v\right):=T\circ f\left(u,v\right)=\left(u,uv,\frac{a_{02}}{2}v^2+a_{11}uv+\frac{a_{20}}{2}u^2\right)^T+O\left(u,v\right)^3,\label{1.2}
\end{eqnarray}
where $a_{02}>0,a_{11},a_{20}$ are constants and $O\left(u,v\right)^3$ is the term whose degree is greater than or equal to $3$ with respect to $\left(u,v\right)$. These constants $a_{ij}$ are independent of the choice of coordinate systems in the domain and isometries in $\mathbb{R}^3$. Such an oriented coordinate system is called a \textit{canonical coordinate system}. It was then proved by Hasegawa, Honda, Naokawa, Umehara, and Yamada in \cite{3} that these constants $a_{ij}$ are intrinsic, i.e., $a_{ij}$ can be expressed only by the coefficients and their partial derivatives of the first fundamental form of $\widetilde{f}$. Furthermore, Hasegawa, Honda, Naokawa, Saji, Umehara, and Yamada defined positive semi-definite metrics called Whitney metrics in \cite{2} and reformulated these constants $a_{ij}$ as invariants of the isolated singular points of Whitney metrics.

Now, the classical Gauss-Bonnet theorem for a compact oriented surface $M$ with boundary asserts that the sum of the integral of the Gaussian curvature $K$ of $M$ and the integral of the geodesic curvature $\kappa_g$ of the boundary $\partial M$ of $M$ is equal to the Euler characteristic $\chi\left(M\right)$:
\begin{equation}
\int_MKdA+\int_{\partial M}\kappa_gds=2\pi\chi\left(M\right),\label{1.3}
\end{equation}
where $dA$ (resp. $ds$) is the area form on $M$ (resp. the arc-length measure of $\partial M$). In \cite{4}, Kuiper derived a Gauss-Bonnet type formula for surfaces without boundary that admit only cross cap singularities and showed that cross cap singularities do not affect the classical Gauss-Bonnet theorem. Then, as a generalization of Kuiper's formula, Hasegawa, Honda, Naokawa, Saji, Umehara, and Yamada, in \cite{2}, derived a Gauss-Bonnet type formula for a Whitney metric on surfaces without boundary and showed that singularities of a Whitney metric also do not affect the classical Gauss-Bonnet theorem.

In this paper, we prove the boundedness of the geodesic curvature of a regular curve passing through a singularity of a Whitney metric (see Theorem $\ref{thm3.1}$). We then generalize the Gauss-Bonnet type formula by Hasegawa et al. to the case of surfaces with boundary, and show that the singularities of a Whitney metric do not affect the formula $(\ref{1.3})$ (see Theorem $\ref{thm4.1}$).

This paper consists of the following: In section $\ref{sec2}$, we briefly review the theory of intrinsic cross cap singularities, which are singularities of a Whitney metric, and West-type coordinate systems according to \cite{2}. In section $\ref{sec3}$, we prove the boundedness of the geodesic curvature measure of a regular curve passing through an intrinsic cross cap singularity. To this purpose, we derive the representation of the geodesic curvature using the coefficients and their partial derivatives of a Whitney metric, and use the formula to mention conditions for the boundedness of the geodesic curvature at an intrinsic cross cap singularity (see $(\ref{3.5})$ and Remark $\ref{rem3.2}$). Finally, in section $\ref{sec4}$, we prove a Gauss-Bonnet type formula for Whitney metrics on surfaces with boundary.

\section{Preliminaries}\label{sec2}

In this section, we review intrinsic cross cap singularities and West-type coordinate systems according to \cite{2}.

Let $M$ be a $2$-manifold and $d\sigma^2:\mathfrak{X}\left(M\right)\times\mathfrak{X}\left(M\right)\to C^\infty\left(M\right)$ a positive semi-definite metrics on $M$, where $\mathfrak{X}\left(M\right)$ is the set of smooth vector fields on $M$ and $C^\infty\left(M\right)$ is the set of smooth real-valued functions on $M$. A point $p\in M$ is a \textit{singular point} of $d\sigma^2$ if $\left(d\sigma^2\right)_p$ is degenerate.

For the sake of brevity, we define
\begin{equation*}
\left\langle\mathbf{X},\mathbf{Y}\right\rangle:=d\sigma^2\left(\mathbf{X},\mathbf{Y}\right)\ \ \ \left(\mathbf{X},\mathbf{Y}\in\mathfrak{X}\left(M\right)\right).\label{2.1}
\end{equation*}
We define a smooth map $\Gamma:\mathfrak{X}\left(M\right)\times\mathfrak{X}\left(M\right)\times\mathfrak{X}\left(M\right)\to C^\infty\left(M\right)$ as
\begin{eqnarray*}
\Gamma\left(\mathbf{X},\mathbf{Y},\mathbf{Z}\right)&:=&\frac{1}{2}\left(\mathbf{X}\left\langle \mathbf{Y},\mathbf{Z}\right\rangle+\mathbf{Y}\left\langle\mathbf{Z},\mathbf{X}\right\rangle-\mathbf{Z}\left\langle\mathbf{X},\mathbf{Y}\right\rangle\right.\nonumber\\
&&\left.+\left\langle\left[\mathbf{X},\mathbf{Y}\right],\mathbf{Z}\right\rangle-\left\langle\left[\mathbf{Y},\mathbf{Z}\right],\mathbf{X}\right\rangle+\left\langle\left[\mathbf{Z},\mathbf{X}\right],\mathbf{Y}\right\rangle\right),\label{2.2}
\end{eqnarray*}
and call it a \textit{Kossowski pseudo-connection}. Then, for each $\mathbf{Y}\in\mathfrak{X}\left(M\right)$ and each $p\in M$, the map
\begin{equation*}
T_pM\times T_pM\to\mathbb{R},\ \left(\mathbf{v}_1,\mathbf{v}_2\right)\mapsto\Gamma\left(\mathbf{V}_1,\mathbf{Y},\mathbf{V}_2\right)\left(p\right)\label{2.3}
\end{equation*}
is a bilinear map, where $\mathbf{V}_j\ \left(j=1,2\right)$ are smooth vector fields on $M$ such that $\mathbf{v}_j=\left(\mathbf{V}_j\right)_p$.

For a point $p\in M$, we define the subspace $\mathcal{N}_p$ of $T_pM$ as
\begin{equation*}
\mathcal{N}_p:=\left\{\mathbf{v}\in T_pM\mid\forall\mathbf{w}\in T_pM,\ \left\langle\mathbf{v},\mathbf{w}\right\rangle=0\right\},\label{2.4}
\end{equation*}
which is called the \textit{null space} at $p$. A non-zero vector belonging to $\mathcal{N}_p$ is called a \textit{null vector} at $p$. Then, for a singular point $p\in M$ of $d\sigma^2$, the map
\begin{equation*}
\widehat{\Gamma}_p:T_pM\times T_pM\times\mathcal{N}_p\to\mathbb{R},\ \widehat{\Gamma}_p\left(\mathbf{v}_1,\mathbf{v}_2,\mathbf{v}_3\right):=\Gamma\left(\mathbf{V}_1,\mathbf{V}_2,\mathbf{V}_3\right)\left(p\right)\label{2.5}
\end{equation*}
is a trilinear map, where $\mathbf{V}_j\ \left(j=1,2,3\right)$ are smooth vector fields on $M$ satisfying $\mathbf{v}_j=\left(\mathbf{V}_j\right)_p$. A singular point $p\in M$ of $d\sigma^2$ is said to be \textit{admissible} if $\widehat{\Gamma}_p=0$. When all singular points of $d\sigma^2$ are admissible, $d\sigma^2$ is called an \textit{admissible metric}.

In the following, we suppose that $d\sigma^2$ is an admissible metric on $M$. Let $\left(U;u,v\right)$ be a local coordinate system centered at a singular point $p\in M$ of $d\sigma^2$ and define a smooth function $\lambda\in C^\infty\left(U\right)$ as
\begin{equation*}
\lambda:=EG-F^2\left(\geq0\right),\label{2.6}
\end{equation*}
where
\begin{equation*}
d\sigma^2=Edu\otimes du+Fdu\otimes dv+Fdv\otimes du+Gdv\otimes dv.\label{2.7}
\end{equation*}
When the Hessian of $\lambda$ does not vanish at $p$, that is,
\begin{eqnarray*}
H_\lambda\left(p\right):=\det\begin{pmatrix}\lambda_{uu}\left(p\right)&\lambda_{uv}\left(p\right)\\\lambda_{uv}\left(p\right)&\lambda_{vv}\left(p\right)\end{pmatrix}\neq0,\label{2.8}
\end{eqnarray*}
$p$ is called an \textit{intrinsic cross cap singularity} of $d\sigma^2$. We note that the definition of an intrinsic cross cap singularity $p$ is independent of the choice of local coordinates of $U$, $p$ is a minima of $\lambda$ and the dimension of $\mathcal{N}_p$ is $1$. If $d\sigma^2$ admits only intrinsic cross cap singularities, then $d\sigma^2$ is called a \textit{Whitney metric}.

Since the null space $\mathcal{N}_p$ at an intrinsic cross cap singularity $p\in M$ is $1$-dimensional, by applying an appropriate affine transformation, we can take a local coordinate system $\left(u,v\right)$ centered at $p$ such that $\left(\frac{\partial}{\partial v}\right)_p\in\mathcal{N}_p$. Such a local coordinate system is called an \textit{adjusted coordinate system}. Then the relation
\begin{equation}
H_\lambda\left(p\right)=4E\left(0,0\right)\Delta\left(0,0\right),\ \ \ \Delta:=\det\begin{pmatrix}E&F_u&F_v\\F_u&\frac{G_{uu}}{2}&\frac{G_{uv}}{2}\\F_v&\frac{G_{uv}}{2}&\frac{G_{vv}}{2}\end{pmatrix}\label{2.9}
\end{equation}
holds. Also,
\begin{equation*}
\alpha_{02}:=\frac{\sqrt{E\left(0,0\right)}\alpha^{\frac{3}{2}}}{\Delta\left(0,0\right)}\ \ \ \left(\alpha:=\frac{\delta_{vv}\left(0,0\right)}{2}\right)\label{2.10}
\end{equation*}
is a positive value independent of the choice of adjusted coordinate systems at $p$. Furthermore, there exist an adjusted coordinate system $\left(u,v\right)$ at an intrinsic cross cap singularity $p\in M$ and two constants $\alpha_{11},\alpha_{20}$ such that
\begin{eqnarray}
E\left(u,v\right)&=&1+\left(\alpha_{20}\right)^2u^2+2\alpha_{11}\alpha_{20}uv\nonumber\\
&&+\left(1+\left(\alpha_{11}\right)^2\right)v^2+O\left(u,v\right)^3,\label{2.11}\\
F\left(u,v\right)&=&\alpha_{11}\alpha_{20}u^2+\left(1+\left(\alpha_{11}\right)^2+\alpha_{02}\alpha_{20}\right)uv\nonumber\\
&&+\alpha_{02}\alpha_{11}v^2+O\left(u,v\right)^3,\label{2.12}\\
G\left(u,v\right)&=&\left(1+\left(\alpha_{11}\right)^2\right)u^2+2\alpha_{02}\alpha_{11}uv\nonumber\\
&&+\left(\alpha_{02}\right)^2v^2+O\left(u,v\right)^3.\label{2.13}
\end{eqnarray}
Such a coordinate system is called a \textit{West-type coordinate system of the second order}. We note that two constants $\left|\alpha_{11}\right|,\alpha_{20}$ are independent of the choice of West-type coordinate systems of the second order at $p$. Furthermore, $\alpha_{11}$ is independent of the choice of West-type coordinate systems of the second order compatible with the orientation of $M$.

\begin{example}[\cite{2}]
\label{ex2.1}

We check that a cross cap singularity $p\in M$ of a smooth map $f:M\to\mathbb{R}^3$ is an intrinsic cross cap singularity of the first fundamental form $d\sigma^2\left(:=df\cdot df\right)$ of $f$. The Kossowski pseudo-connection $\Gamma$ is given by
\begin{equation*}
\Gamma\left(\mathbf{X},\mathbf{Y},\mathbf{Z}\right)=\left(D_{df\left(\mathbf{X}\right)}df\left(\mathbf{Y}\right)\right)\cdot df\left(\mathbf{Z}\right),\label{2.14}
\end{equation*}
where $D:\mathfrak{X}\left(\mathbb{R}^3\right)\times\mathfrak{X}\left(\mathbb{R}^3\right)\to\mathfrak{X}\left(\mathbb{R}^3\right)$ is a Levi-Civita connection on $\mathbb{R}^3$ and ``$\cdot$'' is the Euclidean inner product of $\mathbb{R}^3$. If $\mathbf{Z}_p\in\mathcal{N}_p$, that is, $df_p\left(\mathbf{Z}_p\right)=\mathbf{0}$, then $\Gamma\left(\mathbf{X},\mathbf{Y},\mathbf{Z}\right)\left(p\right)=0$, so $d\sigma^2$ is admissible. Now, taking an adjusted local coordinate system $\left(u,v\right)$ at $p$, $f_u\left(0,0\right),f_{uv}\left(0,0\right),f_{vv}\left(0,0\right)$ are linearly independent by $f_v\left(0,0\right)=\mathbf{0}$ and Whitney's criterion of cross cap singularities (see \cite{6}). Hence,
\begin{eqnarray}
\begin{pmatrix}f_u\\f_{uv}\\f_{vv}\end{pmatrix}\left(f_u,f_{uv},f_{vv}\right)=\begin{pmatrix}E&F_u&F_v\\F_u&\frac{G_{uu}}{2}&\frac{G_{uv}}{2}\\F_v&\frac{G_{uv}}{2}&\frac{G_{vv}}{2}\end{pmatrix}\label{2.15}
\end{eqnarray}
is a regular matrix at $p$, where $E:=f_u\cdot f_u,\ F:=f_u\cdot f_v,\ G:=f_v\cdot f_v$. Thus, by the non-zero determinant of $(\ref{2.15})$ and $(\ref{2.9})$, we see that $p$ is an intrinsic cross cap singularity. Furthermore, the canonical coordinate system at a cross cap singularity is a West-type coordinate system of the second order, and constants $\alpha_{02},\alpha_{11},\alpha_{20}$ of $(\ref{2.11})$-$(\ref{2.13})$ coincides with the coefficients $a_{02},a_{11},a_{20}$ of $(\ref{1.2})$, respectively.
\end{example}

Let $\left(U;u,v\right)$ be an oriented West-type coordinate system of the second order at an intrinsic cross cap singularity $p\in M$. If we set $u\left(r,\theta\right):=r\cos\theta,v\left(r,\theta\right):=r\sin\theta$, then the Gaussian curvature $K$ satisfies
\begin{equation*}
K\left(r,\theta\right)=\frac{\alpha_{02}\left(\alpha_{20}\cos^2\theta-\alpha_{02}\sin^2\theta\right)}{r^2\left(\cos^2\theta+\left(\alpha_{11}\cos\theta+\alpha_{02}\sin\theta\right)^2\right)^2}+\frac{1}{O\left(r\right)},\label{2.16}
\end{equation*}
where $K\left(r,\theta\right):=K\left(u\left(r,\theta\right),v\left(r,\theta\right)\right)$. Then $KdA$ gives a smooth $2$-form with respect to the polar coordinate system $\left(r,\theta\right)$, where $dA:=\sqrt{EG-F^2}du\wedge dv$. In particular, the integral $\int_UKdA$ is well defined.

\section{Boundedness of the geodesic curvature at an intrinsic cross cap singularity}\label{sec3}

In this section, we consider the boundedness of the geodesic curvature measure of a regular curve passing through an intrinsic cross cap singularity. The following assertion holds.

\begin{theorem}
\label{thm3.1}

{\it Let $M$ be a $2$-manifold and $\gamma:\left[0,\varepsilon\right)\to M$ a smooth regular curve starting from an intrinsic cross cap singularity $p\in M$ of an admissible metric $d\sigma^2$. Then the geodesic curvature measure $\kappa_gds$ is a continuous $1$-form on $\left[0,\varepsilon\right)$, where $ds$ is the arc-length measure with respect to $d\sigma^2$.}
\end{theorem}

\begin{proof}

We take a West-type coordinate system of the second order $\left(U;u,v\right)$ centered at $p$ and write the curve $\gamma=\gamma\left(t\right)$ as $\gamma\left(t\right)=\left(u\left(t\right),v\left(t\right)\right)^T$.

First, we express the geodesic curvature $\kappa_g\left(t\right)$ of $\gamma\left(t\right)\ \left(t>0\right)$ using the coefficients $E:=E\left(\gamma\left(t\right)\right),F:=F\left(\gamma\left(t\right)\right),G:=G\left(\gamma\left(t\right)\right)$ of $\left(d\sigma^2\right)_{\gamma\left(t\right)}$. The tangent vector field $\dot{\gamma}\left(t\right)$ of $\gamma\left(t\right)$, the unit normal vector field $\mathbf{n}\left(t\right)$ of $\gamma\left(t\right)$, and the covariant derivative $\nabla_{\dot{\gamma}\left(t\right)}\dot{\gamma}\left(t\right)$ of $\dot{\gamma}\left(t\right)$ can each be written as
\begin{eqnarray}
\dot{\gamma}\left(t\right)&=&\dot{u}\left(t\right)\left(\frac{\partial}{\partial u}\right)_{\gamma\left(t\right)}+\dot{v}\left(t\right)\left(\frac{\partial}{\partial v}\right)_{\gamma\left(t\right)},\label{3.1}\\
\mathbf{n}\left(t\right)&=&\frac{\mathrm{sgn}\left(t\right)}{\sqrt{EG-F^2}\sqrt{\dot{u}\left(t\right)^2E+2\dot{u}\left(t\right)\dot{v}\left(t\right)F+\dot{v}\left(t\right)^2G}}\nonumber\\
&&\times\left\{-\left(\dot{u}\left(t\right)F+\dot{v}\left(t\right)G\right)\left(\frac{\partial}{\partial u}\right)_{\gamma\left(t\right)}\right.\nonumber\\
&&\left.+\left(\dot{u}\left(t\right)E+\dot{v}\left(t\right)F\right)\left(\frac{\partial}{\partial v}\right)_{\gamma\left(t\right)}\right\},\label{3.2}\\
\nabla_{\dot{\gamma}\left(t\right)}\dot{\gamma}\left(t\right)&=&\left\{\ddot{u}\left(t\right)+\frac{\dot{u}\left(t\right)^2\left(GE_u-2FF_u+FE_v\right)}{2\left(EG-F^2\right)}\right.\nonumber\\
&&+\frac{\dot{u}\left(t\right)\dot{v}\left(t\right)\left(GE_v-FG_u\right)}{EG-F^2}\nonumber\\
&&\left.+\frac{\dot{v}\left(t\right)^2\left(2GF_v-GG_u-FG_v\right)}{2\left(EG-F^2\right)}\right\}\left(\frac{\partial}{\partial u}\right)_{\gamma\left(t\right)}\nonumber\\
&&+\left\{\ddot{v}\left(t\right)+\frac{\dot{u}\left(t\right)^2\left(2EF_u-EE_v-FE_u\right)}{2\left(EG-F^2\right)}\right.\nonumber\\
&&+\frac{\dot{u}\left(t\right)\dot{v}\left(t\right)\left(EG_u-FE_v\right)}{EG-F^2}\nonumber\\
&&\left.+\frac{\dot{v}\left(t\right)^2\left(EG_v-2FF_v+FG_u\right)}{2\left(EG-F^2\right)}\right\}\left(\frac{\partial}{\partial v}\right)_{\gamma\left(t\right)},\label{3.3}
\end{eqnarray}
where $\nabla$ is the Levi-Civita connection on $M$. Furthermore, to obtain the expression $(\ref{3.3})$, we used Christoffel's symbol $\Gamma_{ij}^k\ \left(i,j,k=1,2\right)$ given by (see \cite{7})
\begin{eqnarray*}
&&\Gamma_{11}^1:=\frac{GE_u-2FF_u+FE_v}{2\left(EG-F^2\right)},\ \ \ \Gamma_{11}^2:=\frac{2EF_u-EE_v-FE_u}{2\left(EG-F^2\right)},\nonumber\\
&&\Gamma_{12}^1=\Gamma_{21}^1:=\frac{GE_v-FG_u}{2\left(EG-F^2\right)},\ \ \ \Gamma_{12}^2=\Gamma_{21}^2:=\frac{EG_u-FE_v}{2\left(EG-F^2\right)},\nonumber\\
&&\Gamma_{22}^1:=\frac{2GF_v-GG_u-FG_v}{2\left(EG-F^2\right)},\ \ \ \Gamma_{22}^2:=\frac{EG_v-2FF_v+FG_u}{2\left(EG-F^2\right)}.\label{3.4}
\end{eqnarray*}
Then, by the equations $(\ref{3.1})$-$(\ref{3.3})$, the geodesic curvature $\kappa_g\left(t\right)$ of $\gamma\left(t\right)\ \left(t>0\right)$ is given by
\begin{eqnarray}
\kappa_g\left(t\right)&:=&\frac{\left\langle\nabla_{\dot{\gamma}\left(t\right)}\dot{\gamma}\left(t\right),\mathbf{n}\left(t\right)\right\rangle}{\left|\dot{\gamma}\left(t\right)\right|^2}\nonumber\\
&=&\mathrm{sgn}\left(t\right)\left\{\frac{\left(\dot{u}\left(t\right)\ddot{v}\left(t\right)-\dot{v}\left(t\right)\ddot{u}\left(t\right)\right)\sqrt{EG-F^2}}{\left(\dot{u}\left(t\right)^2E+2\dot{u}\left(t\right)\dot{v}\left(t\right)F+\dot{v}\left(t\right)^2G\right)^{\frac{3}{2}}}\right.\nonumber\\
&&+\frac{\dot{u}\left(t\right)^3\left(2EF_u-EE_v-FE_u\right)}{2\sqrt{EG-F^2}\left(\dot{u}\left(t\right)^2E+2\dot{u}\left(t\right)\dot{v}\left(t\right)F+\dot{v}\left(t\right)^2G\right)^{\frac{3}{2}}}\nonumber\\
&&-\frac{\dot{v}\left(t\right)^3\left(2GF_v-GG_u-FG_v\right)}{2\sqrt{EG-F^2}\left(\dot{u}\left(t\right)^2E+2\dot{u}\left(t\right)\dot{v}\left(t\right)F+\dot{v}\left(t\right)^2G\right)^{\frac{3}{2}}}\nonumber\\
&&+\frac{\dot{u}\left(t\right)^2\dot{v}\left(t\right)\left(2EG_u-3FE_v-GE_u+2FF_u\right)}{2\sqrt{EG-F^2}\left(\dot{u}\left(t\right)^2E+2\dot{u}\left(t\right)\dot{v}\left(t\right)F+\dot{v}\left(t\right)^2G\right)^{\frac{3}{2}}}\nonumber\\
&&\left.-\frac{\dot{u}\left(t\right)\dot{v}\left(t\right)^2\left(2GE_v-3FG_u-EG_v+2FF_v\right)}{2\sqrt{EG-F^2}\left(\dot{u}\left(t\right)^2E+2\dot{u}\left(t\right)\dot{v}\left(t\right)F+\dot{v}\left(t\right)^2G\right)^{\frac{3}{2}}}\right\},\label{3.5}
\end{eqnarray}
where $\left|\xi\right|:=\sqrt{\left\langle\xi,\xi\right\rangle}$. Furthermore, the arc length measure $ds:=\left|\dot{\gamma}\left(t\right)\right|dt$ is expressed as
\begin{eqnarray}
ds=\sqrt{\dot{u}\left(t\right)^2E+2\dot{u}\left(t\right)\dot{v}\left(t\right)F+\dot{v}\left(t\right)^2G}dt.\label{3.6}
\end{eqnarray}

Now, we examine the boundedness of the geodesic curvature measure $\kappa_gds$ at $p$ using formulas $(\ref{3.5})$ and $(\ref{3.6})$. First, we consider the case where the curve $\gamma\left(t\right)$ is transversal to the null space $\mathcal{N}_p$ at $p$. We parameterize $\gamma\left(t\right)$ as
\begin{equation}
u\left(t\right)=t,\ \ \ v\left(t\right)=\sum_{i=1}^3\frac{v^{\left(i\right)}\left(0\right)}{i!}t^i+O\left(t\right)^4.\label{3.7}
\end{equation}
Computing $E,F,G$ and their partial derivatives with respect to $(\ref{3.7})$ and substituting these expressions into $(\ref{3.5})$, we obtain
\begin{eqnarray*}
\kappa_g\left(t\right)=\frac{2\dot{v}\left(0\right)+\left(\alpha_{11}+\alpha_{02}\dot{v}\left(0\right)\right)\left(\alpha_{20}+\dot{v}\left(0\right)\left(2\alpha_{11}+\alpha_{02}\dot{v}\left(0\right)\right)\right)}{\sqrt{1+\left(\alpha_{11}+\alpha_{02}\dot{v}\left(0\right)\right)^2}}+O\left(t\right).\label{3.8}
\end{eqnarray*}
Thus, we see that $\kappa_g$ is a continuous function on $\left[0,\varepsilon\right)$. In particular, $\kappa_gds$ is a continuous $1$-form on $\left[0,\varepsilon\right)$:
\begin{eqnarray*}
\kappa_g\left(t\right)ds=\left\{\frac{2\dot{v}\left(0\right)+\left(\alpha_{11}+\alpha_{02}\dot{v}\left(0\right)\right)\left(\alpha_{20}+\dot{v}\left(0\right)\left(2\alpha_{11}+\alpha_{02}\dot{v}\left(0\right)\right)\right)}{\sqrt{1+\left(\alpha_{11}+\alpha_{02}\dot{v}\left(0\right)\right)^2}}+O\left(t\right)\right\}dt.\label{3.9}
\end{eqnarray*}

Next, we consider the case where the curve $\gamma\left(t\right)$ is tangent to the null space $\mathcal{N}_p$ at $p$. We parameterize $\gamma\left(t\right)$ as
\begin{equation}
u\left(t\right)=\sum_{i=2}^4\frac{u^{\left(i\right)}\left(0\right)}{i!}t^i+O\left(t\right)^5,\ \ \ v\left(t\right)=t.\label{3.10}
\end{equation}
Computing $E,F,G$ and their partial derivatives with respect to $(\ref{3.10})$ and substituting these expressions into $(\ref{3.5})$, we obtain
\begin{eqnarray}
\kappa_g\left(t\right)&=&\frac{8\left(\alpha _{02}\right)^2 \alpha _{11} \ddot{u}\left(0\right)^2-2 u^{\left(3\right)}\left(0\right) \left(\alpha _{02}\right)^3+u^{\left(3\right)}\left(0\right) \alpha _{02} \ddot{u}\left(0\right)^2-\alpha _{11} \ddot{u}\left(0\right)^4}{2 t \left(\left(\alpha _{02}\right)^2+\ddot{u}\left(0\right)^2\right)^{\frac{5}{2}}}\nonumber\\
&&+\frac{1}{24\alpha _{02} \left(\left(\alpha _{02}\right)^2+\ddot{u}\left(0\right)^2\right)^{\frac{7}{2}}}\left\{\ddot{u}\left(0\right) \left(-12 \left(\alpha _{02}\right)^6\right.\right.\nonumber\\
&&+18\left( \alpha _{02}\right)^4 \ddot{u}\left(0\right)^2 \left(2 \alpha _{02} \alpha _{20}-15 \left(\alpha _{11}\right)^2+3\right)\nonumber\\
&&+9 \left(\alpha _{02}\right)^2 \ddot{u}\left(0\right)^4 \left(4 \alpha _{02} \alpha _{20}+15\left( \alpha _{11}\right)^2+7\right)\nonumber\\
&&+2 u^{\left(3\right)}\left(0\right) \alpha _{02} \alpha _{11} \left(70 \left(\alpha _{02}\right)^4-64 \left(\alpha _{02}\right)^2 \ddot{u}\left(0\right)^2+\ddot{u}\left(0\right)^4\right)\nonumber\\
&&\left.+45 u^{\left(3\right)}\left(0\right)^2\left( \alpha _{02}\right)^4-3 \ddot{u}\left(0\right)^6\right)\nonumber\\
&&\left.-12 u^{\left(4\right)}\left(0\right) \left(\alpha _{02}\right)^4 \left(\left(\alpha _{02}\right)^2+\ddot{u}\left(0\right)^2\right)\right\}+O\left(t\right)^3.\label{3.11}
\end{eqnarray}
Then the geodesic curvature measure $\kappa_gds$ is given by
\begin{eqnarray*}
\kappa_g\left(t\right)ds=\left\{\frac{8\left( \alpha _{02}\right)^2 \alpha _{11}\ddot{u}\left(0\right)^2-2 u^{\left(3\right)}\left(0\right)\left( \alpha _{02}\right)^3+u^{\left(3\right)}\left(0\right) \alpha _{02}\ddot{u}\left(0\right)^2-\alpha _{11}\ddot{u}\left(0\right)^4}{2  \left(\left(\alpha _{02}\right)^2+\ddot{u}\left(0\right)^2\right)^2}+O\left(t\right)\right\}dt.\label{3.12}
\end{eqnarray*}
Hence, $\kappa_gds$ is a continuous $1$-form on $\left[0,\varepsilon\right)$.\qed
\end{proof}

\begin{remark}
\label{rem3.2}

In general, if a amooth regular curve $\gamma\left(t\right)\ \left(t\in\left[0,\varepsilon\right)\right)$ starting from an intrinsic cross cap singularity $p\in M$ is tangent to a null space $\mathcal{N}_p$ at $p$ and
\begin{equation*}
8\left(\alpha _{02}\right)^2 \alpha _{11} \ddot{u}\left(0\right)^2-2 u^{\left(3\right)}\left(0\right) \left(\alpha _{02}\right)^3+u^{\left(3\right)}\left(0\right) \alpha _{02} \ddot{u}\left(0\right)^2-\alpha _{11} \ddot{u}\left(0\right)^4=0,
\end{equation*}
the geodesic curvature $\kappa_g\left(t\right)$ is bounded at $p$ by $(\ref{3.11})$.
\end{remark}

\begin{remark}
\label{rem3.3}

Setting $u\left(t\right)=t,\ v\left(t\right)=0$ in $(\ref{3.7})$, the formula for the geodesic curvature measure $\kappa_gds$ with respect to polar coordinates is derived in the proof of \cite[Theorem $5.1$]{2}.
\end{remark}

\begin{example}
\label{ex3.4}

We consider a standard cross cap $(\ref{1.1})$. Then, by Example $\ref{ex2.1}$, the coordinate system $\left(u,v\right)$ is a West-type coordinate system of the second order, and the null space at the singular point $\left(0,0\right)$ is the $v$-axis. Now, we consider the following four types of curves in the $uv$-plane:
\begin{eqnarray*}
&&\gamma_1\left(t\right):=\left(t,t\right)^T,\ \ \ \gamma_2\left(t\right):=\left(\frac{t^2}{2},t\right)^T,\nonumber\\
&&\gamma_3\left(t\right):=\left(\frac{t^3}{6},t\right)^T,\ \ \ \gamma_4\left(t\right):=\left(\frac{t^4}{24},t\right)^T.\label{3.13}
\end{eqnarray*}
The curve $\gamma_1$ is transversal to the $v$-axis at $\left(0,0\right)$ and $\gamma_i\ \left(i=2,3,4\right)$ are tangent to the $v$-axis at $\left(0,0\right)$. The images of these curves $\widehat{\gamma}_i:=f_0\circ\gamma_i\ \left(i=1,2,3,4\right)$ are expressed as (see Figure $\ref{fig3.1}$)
\begin{eqnarray*}
&&\widehat{\gamma}_1\left(t\right)=\left(t,t^2,t^2\right)^T,\ \ \ \widehat{\gamma}_2\left(t\right)=\left(\frac{t^2}{2},\frac{t^3}{2},t^2\right)^T,\nonumber\\
&&\widehat{\gamma}_3\left(t\right)=\left(\frac{t^3}{6},\frac{t^4}{6},t^2\right)^T,\ \ \ \widehat{\gamma}_3\left(t\right)=\left(\frac{t^4}{24},\frac{t^5}{24},t^2\right)^T.\label{3.14}
\end{eqnarray*}
The geodesic curvatures $\kappa_g^i$ of $\widehat{\gamma}_i\ \left(i=1,2,3,4\right)$ are calculated as
\begin{eqnarray*}
\kappa_g^1\left(t\right)&=&\frac{6}{\sqrt{4t^2+5}\left(8t^2+1\right)^{\frac{3}{2}}}\to\frac{6}{\sqrt{5}}\ \ \ \left(t\to0\right),\nonumber\\
\kappa_g^2\left(t\right)&=&-\mathrm{sgn}\left(t\right)\frac{84}{\left(9 t^2+20\right)^{\frac{3}{2}} \sqrt{17 t^2+16}}\to\mp\frac{21}{40 \sqrt{5}}\ \ \ \left(t\to\pm0\right),\nonumber\\
\kappa_g^3\left(t\right)&=&\mathrm{sgn}\left(t\right)\frac{72 \left(t^4-96 t^2-36\right)}{t \sqrt{t^4+144 t^2+144} \left(16 t^4+9 t^2+144\right)^{\frac{3}{2}}},\nonumber\\
\kappa_g^4\left(t\right)&=&\mathrm{sgn}\left(t\right)\frac{96 \left(5 t^6-8640 t^2-4608\right)}{\sqrt{t^6+2304 t^2+2304} \left(25 t^6+16 t^4+2304\right)^{\frac{3}{2}}}\nonumber\\
&&\to\mp\frac{1}{12}\ \ \ \left(t\to\pm0\right).\label{3.15}
\end{eqnarray*}
We note that $\kappa_g^3\left(t\right)$ diverges as $t$ tends to zero. Furthermore, the geodesic curvature measures $\kappa_g^ids_i\ \left(ds^i:=\sqrt{\widehat{\gamma}_i\left(t\right)\cdot\widehat{\gamma}_i\left(t\right)}dt\right)$ of the curves $\widehat{\gamma}_i$ are computed as
\begin{eqnarray*}
\kappa_g^1\left(t\right)ds_1&=&\frac{6}{\sqrt{4 t^2+5} \left(8 t^2+1\right)}dt\to\frac{6}{\sqrt{5}}dt\ \ \ \left(t\to0\right),\nonumber\\
\kappa_g^2\left(t\right)ds_2&=&-\frac{42 t }{\left(9 t^2+20\right) \sqrt{17 t^2+16}}dt\to0dt\ \ \ \left(t\to0\right),\nonumber\\
\kappa_g^3\left(t\right)ds_3&=&\frac{12 \left(t^4-96 t^2-36\right)  }{\sqrt{t^4+144 t^2+144} \left(16 t^4+9 t^2+144\right)}dt\nonumber\\
&&\to-\frac{1}{4}dt\ \ \ \left(t\to0\right),\nonumber\\
\kappa_g^4\left(t\right)ds_4&=&\frac{4 t \left(5 t^6-8640 t^2-4608\right)  }{\sqrt{t^6+2304 t^2+2304} \left(25 t^6+16 t^4+2304\right)}dt\nonumber\\
&&\to0dt\ \ \ \left(t\to0\right).\label{3.16}
\end{eqnarray*}
Hence, the $1$-forms $\kappa_g^ids_i$ are continuous.
\begin{figure}[ht]
\centering
\includegraphics[width=2.5cm]{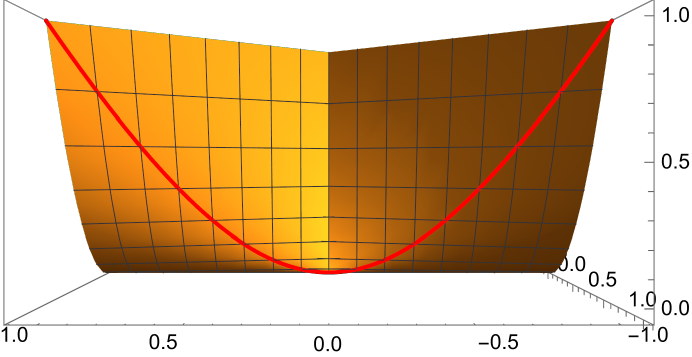}\ \ \ \includegraphics[width=2.5cm]{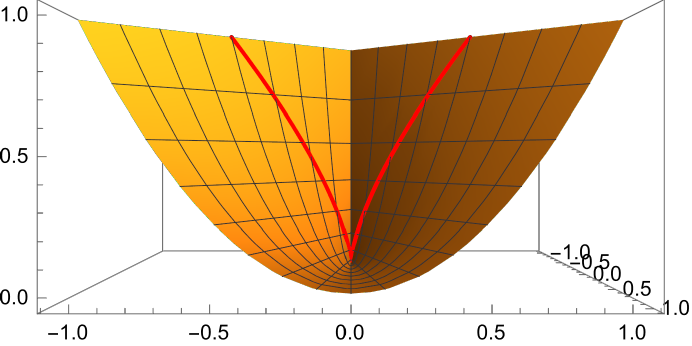}\ \ \ \includegraphics[width=2.5cm]{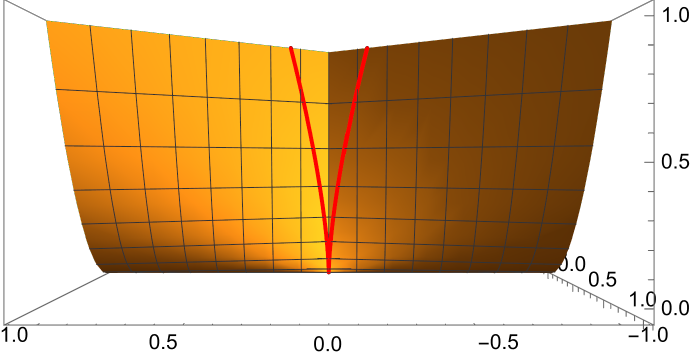}\ \ \ \includegraphics[width=2.5cm]{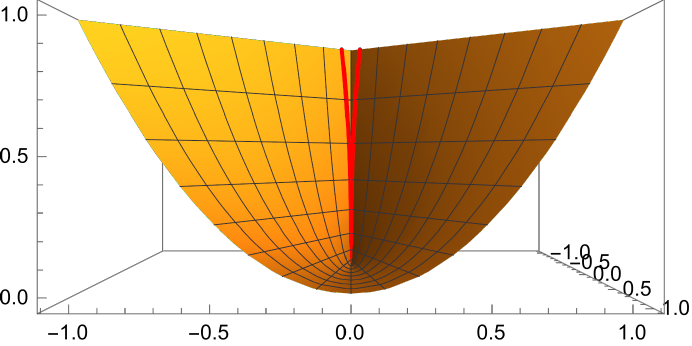}
\caption{The red lines drawn on the cross caps represent $\widehat{\gamma}_i \left(i=1,2,3,4\right)$ from left to right.}
\label{fig3.1}
\end{figure}
\end{example}

\section{The Gauss-Bonnet type formula for Whitney metrics on surfaces with boundary}\label{sec4}

In this section, we prove a Gauss-Bonnet type formula for Whitney metrics on surfaces with boundary, which is a generalization of \cite[Theorem $5.1$]{2} to the case of surfaces with boundary.

\begin{theorem}
\label{thm4.1}

{\it Let $M$ be a compact oriented $2$-manifold with boundary $\partial M$ and $d\sigma^2$ a Whitney metric on $M$. Then
\begin{equation*}
\int_MKdA+\int_{\partial M}\kappa_gds=2\pi\chi\left(M\right)\label{4.1}
\end{equation*}
holds, that is, intrinsic cross cap singularities do not affect the classical Gauss-Bonnet theorem $(\ref{1.3})$}
\end{theorem}

\begin{proof}

The formula for the case where a singular point $p\in M$ of $d\sigma^2$ is in the interior $M\backslash\partial M$ of $M$ can be proved exactly the same as in \cite[Theorem $5.1$]{2}. We prove the case where $p$ is on $\partial M$.

Since $p$ is an intrinsic cross cap singularity, we can take an oriented West-type coordinate system of the second order $\left(U;u,v\right)$ centered at $p$. Let $D\left(r\right)$ be a closed disc contained in $U$ with center $p$ and radius $r>0$, and $\left\{P_1,P_2\right\}$ the intersection of the boundary $S\left(r\right)$ of $D\left(r\right)$ with $\partial M$:
\begin{equation*}
D\left(r\right):=\left\{\left(u,v\right)\in U\mid u^2+v^2\leq r^2\right\},\ \left\{P_1,P_2\right\}=S\left(r\right)\cap\partial M.\label{4.2}
\end{equation*}
Then, by a proof similar to the classical Gauss-Bonnet theorem, we obtain
\begin{eqnarray}
\int_{D\left(r\right)}KdA+\int_{S\left(r\right)}\kappa_gds=2\pi-\sum_{l=1,2}\left(\pi-\angle{P_l}\right),\label{4.3}
\end{eqnarray}
where $\angle P_l$ is the interior angle at $P_l$. Here, we note that the integrals on the left-hand side of $(\ref{4.3})$ are well defined, respectively (see the explanation under Example $\ref{ex2.1}$ and Theorem $\ref{thm3.1}$).\qed
\end{proof}

\section*{Declarations}

\subsection*{Ethics approval and consent to participate}

Not applicable

\subsection*{Consent for publication}

Not applicable.

\subsection*{Availability of data and materials}

Not applicable.

\subsection*{Competing interests}

The authors declare that they have no competing interests.

\subsection*{Funding}

This work was supported by the Hokkaido University Ambitious Doctoral Fellowship (Information Science and AI).

\subsection*{Authors' contributions}

K.H. designed the study and worked on the manuscript.

\subsection*{Acknowledgements}

The author thanks Go-o Ishikawa for fruitful discussions and valuable comments.

\begin{itemize}
\item[]Kyoya Hashibori,\\Department of Mathematics,\\Graduate School of Science,\\Hokkaido University,\\Kita-10, Nishi-8, Kita-ku, Sapporo, Hokkaido, 060-0810, Japan\\e-mail: hashibori.kyoya.a7@elms.hokudai.ac.jp
\end{itemize}

\end{document}